\documentclass{amsart}
\usepackage{tikz}

\usepackage{amsmath,amsfonts,amsthm,amssymb,amscd, verbatim, graphicx,textcomp, hyperref}
\usepackage[margin=2cm]{geometry}

\usepackage{pinlabel}
\usepackage{lscape}
\usepackage{latexsym}
\usepackage{multirow}

\newtheorem{Thm}{Theorem}[section]
\newtheorem{Cor}[Thm]{Corollary}
\newtheorem{Conj}[Thm]{Conjecture}

\newtheorem{Qu}[Thm]{Question}
\newtheorem{Lem}[Thm]{Lemma}

\newtheorem*{thma}{Theorem A}
\newtheorem*{thmb}{Theorem B}
\newtheorem*{thmc}{Theorem C}

\theoremstyle{definition}
\newtheorem{Def}[Thm]{Definition}

\newtheorem{Ex}[Thm]{Example}

\theoremstyle{remark}

\newtheorem{remark}[Thm]{Remark}

\numberwithin{equation}{section}

\newcommand{\Aut}{\operatorname{Aut}}

\newcommand{\stab}{\operatorname{stab}}
\newcommand{\Id}{\operatorname{Id}}

\newcommand{\D}{\mathcal{D}}
\newcommand{\De}{\mathcal{D}}

\newcommand{\Sym}{\operatorname{Sym}}

\newcommand{\supp}{\operatorname{supp}}

\newcommand{\PSp}{\operatorname{PSp}}
\newcommand{\Oo}{\operatorname{O}}
\renewcommand{\Gamma}{\varGamma}
\renewcommand{\epsilon}{\varepsilon}
\renewcommand{\bar}{\overline}

\renewcommand{\leq}{\leqslant}
\renewcommand{\geq}{\geqslant}

\newcommand{\B}{\mathcal{B} }

\renewcommand{\B}{\mathcal{B}}

\renewcommand{\L}{\mathcal{L}}

\newcommand{\W}{\textbf{W}}

\renewcommand{\L}{\mathcal{L}}

\renewcommand{\H}{\mathcal{H}}

\begin{document}


\title{Generating groups using hypergraphs}
 

\author{Nick Gill}
\address{Department of Mathematics, University of South Wales, U.K.}
\email{nick.gill@southwales.ac.uk}
\author{Neil I. Gillespie}
\address{Heilbronn Institute for Mathematical Research, Department of Mathematics, University of Bristol, U.K.}
\email{neil.gillespie@bristol.ac.uk}
\author{Anthony Nixon}
\address{Department of Mathematics and Statistics, Lancaster University, U.K.}
\email{a.nixon@lancaster.ac.uk}
\author{Jason Semeraro}
\address{Heilbronn Institute for Mathematical Research, Department of Mathematics, University of Bristol, U.K.}
\email{js13525@bristol.ac.uk}



\begin{abstract}
To a set $\B$ of 4-subsets of a set $\Omega$ of size $n$ we introduce an invariant called the `hole stabilizer' which generalises a construction of Conway, Elkies and Martin of the Mathieu group $M_{12}$ based on Loyd's `15-puzzle'. It is shown that hole stabilizers may be regarded as objects inside an objective partial group (in the sense of Chermak). We classify pairs $(\Omega,\B)$ with a trivial hole stabilizer, and determine all hole stabilizers associated to $2$-$(n,4,\lambda)$ designs with $\lambda \leq 2$. 
\end{abstract}



\maketitle

\section{Introduction}

In the beautiful papers \cite{Co1} and \cite{Co}, John Conway (resp. Conway, Elkies and Martin) studies a construction of the Mathieu group $M_{12}$ that is reminiscent of the classic 15-puzzle played with tiles on a $4 \times 4$ grid. In Conway's construction, the grid is replaced by the projective plane $\mathbb{P}_3,$ and the $15$ tiles are replaced by $12$ counters which sit on all but one of the 13 points of $\mathbb{P}_3.$ 

The point without a counter is known as the \textit{hole}. Given a point $p$ with a counter, one performs an ``elementary move'' by placing that counter on the hole $h$ and interchanging the two counters on $q$ and $r$ where $\{p,h,q,r\}$ is the line containing $p$ and $h$. A sequence of such moves is \textit{closed} if that sequence returns the hole $h$ to its initial location. The set of all closed sequences forms a group $\pi_h$ under concatenation and it is shown in \cite{Co1} that $\pi_h \cong M_{12}$ for all holes $h$.

We study a generalised version of this set-up, where $\mathbb{P}_3$ is replaced by a \emph{$4$-hypergraph}, i.e.\ by a pair $\De:=(\Omega,\B)$, and where $\Omega$ is a finite set of size $n$ and $\B$ is a finite multiset of subsets of $\Omega$ (called lines), each of size $4$.  

For any $4$-hypergraph $\De:=(\Omega,\B)$ one generates permutations analogously to before, by placing counters on $n-1$ of the $n$ points in $\Omega$, nominating the remaining point $h$ as the \textit{hole}, and considering closed sequences of elementary moves. An elementary move at a point $p$ is possible if and only if $p$ and $h$ are collinear; it involves placing the counter at $p$ on the hole $h$ and interchanging the counters on $q$ and $r$ for \textit{each} of the lines $\{p,h,q,r\}$.

It is clear that, a priori, an elementary move at a point $p$ may result in different permutations of the set of counters, depending on the order in which one moves through the lines containing $p$ and $h$. In order to ensure that an elementary move determines a unique permutation of the set of counters we restrict our attention to those $4$-hypergraphs that are \emph{pliable}, i.e.\ for which, whenever three points are all contained in two lines, the two lines contain exactly the same points.

We define the \textit{puzzle set} $\L_\De$ to be the subset of $\Sym(\Omega)$ consisting of all move sequences and $\pi_\infty(\De)$ to be the \textit{hole stabilizer} of all closed move sequences for each $\infty \in \Omega$.
Since an arbitrary pair of move sequences cannot be combined to form another move sequence, $\L_\De$ does not form a group, but it does have the structure of a \textit{partial group}. For a definition of a partial group (and of any other unexplained terminology thus far), we refer the reader to \S\ref{s:background}.

\subsection{Main results and structure of the paper}\label{sub:main}

Our first main result, Theorem A, is a direct generalization of a result from \cite{Co}; it connects pliable $4$-hypergraphs with partial groups. Indeed, it does a little more, for under certain conditions, the partial groups we consider here are \emph{objective}; see \S\ref{s:background} for a full definition of this notion. Theorem A is proved in \S\ref{s:puzz}.

\begin{thma}
Let $\De:=(\Omega,\B)$ be a pliable $4$-hypergraph. Then the quadruple $(\L_\De, \D,\Pi,(-)^{-1})$ forms a partial group.

Suppose, in addition, that, for any two points $x,y\in\Omega$, there is a sequence of distinct points $x_0,\dots,x_n$ such that $x=x_0$, $y=x_n$ and, for all $i=1,\dots, n$, the points $x_{i-1}$ and $x_i$ are collinear. Then $(\L_{\De},\Delta)$ is an objective partial group where $$\Delta:=\{\pi_x(\De) \mid x \in \Omega\}.$$ In particular, $\pi_x(\De) \cong \pi_y(\De)$ for all $x,y \in \Omega$.
\end{thma}


$(\L_\De, \D,\Pi,(-)^{-1})$ is defined in \S\ref{s:puzz}, specifically at \eqref{e:fulldef1} and \eqref{e:fulldef2}. We remark that two points $x,y\in\Omega$ are called \emph{collinear} if there is a line in $\B$ that contains them both.

In \S\ref{exsext} we study various examples of pliable $4$-hypergraphs. We also commence a study of those hypergraphs that are $2$-$(n,4,\lambda)$ designs (see \S\ref{s:background} for a definition) which will be an important theme of the rest of the paper. 

Designs are particularly interesting to us because they satisfy the extra supposition given in Theorem A. Hence, first of all, the associated partial groups are all objective. What is more, up to isomorphism, the group $\pi_\infty(\De)$ is unique, i.e.\ it does not depend on the choice of the point $\infty$. 

In \S\ref{exsext} we are able to construct various groups as $\pi_\infty(\De)$, for a point $\infty$ in a design $\De$. These groups are listed in Table 1, and include the almost simple group $\mathrm{PSp}_4(3):2$ (as a primitive subgroup of $S_{27}$), $S_6$ (as a primitive subgroup of $S_{15}$) and $S_8$ (as a primitive subgroup of $S_{35}$). In \S\ref{s: subsequent} we also briefly discuss a recently constructed infinite family of examples.

In \S\ref{sec:theorem b} we present our second main result, Theorem B, which characterises $4$-hypergraphs with a trivial hole stabilizer. First, a definition: a \textit{Boolean quadruple system of order $2^k$}, is a $4$-hypergraph $\De=(\Omega,\B)$ where $\Omega:=\mathbb{F}_2^k$ for some $k$, and $\B$ consists of all sets of four vectors whose sum is $\underline{0}$.

\begin{thmb}
Let $\De=(\Omega,\B)$ be a simple pliable $4$-hypergraph and let $\L:=\L_\De$. Then the following are equivalent:
\begin{itemize}
\item[(a)] $\pi_{\infty}(\De)$ is trivial for all $\infty \in \Omega$;
\item[(b)] $\De$ is the Boolean quadruple system of order $2^k$ for some $k > 0$.
\end{itemize}
\end{thmb}


%

Our third result, Theorem C, is proved in \S\ref{s: thmc}, using results from \S\ref{s: lambda12}. Theorem C gives a full classification of the structure of $\pi_\infty(\De)$ for all simple pliable $2$-$(n,4,\lambda)$ designs with $\lambda < 3$. In fact, the notion of a ``simple pliable design'' has already appeared in the literature, where it is known as a \emph{supersimple design}. We will use this terminology in what follows.

\begin{thmc}
Let $n \geq 7$ and $\De=(\Omega,\B)$ be a supersimple $2$-$(n,4,\lambda)$ design. If $\lambda < 3$ then one of the following holds:
\begin{itemize}
\item[(a)] $\lambda=1$, $n \equiv 1,4 \pmod{12}$ and  for all $\infty \in \Omega$
$$\pi_{\infty}(\De)\cong\left\{
	\begin{array}{ll}
		M_{12},  & \mbox{if } n = 13; \\
		A_{n-1}, & \mbox{otherwise. } 
	\end{array}
\right.$$
\item[(b)]  $\lambda=2$, $n \equiv 1 \pmod 3$ and for all $\infty \in \Omega$
$$\pi_{\infty}(\De)\cong\left\{
	\begin{array}{ll}
		S_3 \wr S_2,  & \mbox{if } n = 10; \\
		S_{n-1}, & \mbox{otherwise. } 
	\end{array}
\right.$$
\end{itemize}
\end{thmc}

Since  $|\L_\De| \geq n\cdot |\pi_\infty(\De)|$ (see Lemma \ref{lem1} (c)), Theorem C shows that $\L_\De$ consists of all elements in the symmetric or alternating group unless $(n,\lambda)=(13,1)$ or $(10,2)$. In the former case $\L_\De$ contains the set $M_{13}$, studied in \cite{Co}; in the latter case $\L_\De$ is equal (as a set) to a primitive subgroup of $S_{10}$ that is isomorphic to $S_6$ (see Example \ref{ex2}).

In the process of proving Theorem C we present various results  concerning the structure of $\pi_\infty(\De)$ for arbitrary supersimple $2$-$(n,4,\lambda)$ designs $\De:=(\Omega,\B)$ (i.e.\ without the restriction $\lambda\leq 2$). These results, which are presented in \S\ref{s: lambda12}, give a rough picture of how the structure of $\pi_\infty(\De)$ behaves as a permutation group (i.e.\ in its natural embedding in $\Sym(\Omega\backslash\{\infty\})$. 
It turns out that, roughly speaking, if we fix $\lambda$ and allow $n$ to increase, $\pi_\infty(\De)$ moves through the following states:
\[
 \textrm{trivial} \longrightarrow \textrm{intransitive} \longrightarrow \begin{array}{l}
                                                                         \textrm{transitive} \\ \textrm{imprimitive}
                                                                        \end{array} \longrightarrow \textrm{primitive}
\longrightarrow A_{n-1} \textrm{ or } S_{n-1}.
\]

In \S\ref{s: ext} we present a conjecture relating to this observed behaviour of $\pi_\infty(\De)$, along with a number of open questions, and avenues for future work.

\subsection{Subsequent work}\label{s: subsequent}

The work in the current paper has been recently extended by the first, second and fourth authors in two directions \cite{ggs}: Firstly, the aforementioned conjecture concerning the behaviour of $\pi_\infty(\De)$ is proved. 

Secondly, an infinite family of designs has been constructed for which the associated hole stabilizers are primitive but neither symmetric nor alternating. Specifically it is proved that for all $m \geq 2$ the groups $\Oo_{2m}^+(2)$ and  $\Oo_{2m}^-(2)$ arise as hole stabilizers. We remark that the group $S_3\wr S_2$ occurring in Theorem C lies in this family thanks to the isomorphism  $S_3\wr S_2\cong \Oo_4^+(2)$. This family also includes the two groups ${\PSp_4(3):2}$ (with $n=28$) and $S_8$ (with $n=36$) that are listed in Table 1 of \S\ref{exsext}. (Note that ${\PSp_4(3):2}\cong \Oo_6^-(2)$ and $S_8\cong \Oo_6^+(2)$.)

This family of examples has a remarkable extra property that is foreshadowed in the remarks following Theorem C, and in the relevant entries of Table 1: if $\De$ is a design from this family with hole stabilizer $\Oo_{2m}^\varepsilon(2)$ (for $\varepsilon$ either $-$ or $+$) then the set $\L_\De$ coincides as a set with a primitive subgroup ${\rm Sp}_{2m}(2)$ in $S_{n-1}$. We remark that, to this point, the only primitive hole stabilizer that we have found for which $\L_\De$ is {\it not} a group is the group $M_{12}$ in the construction of Conway.

The aforementioned remarkable property has been studied subsequently by Cheryl E.~Praeger together with the first, second and fourth authors \cite{ggps}. They were able to give strong information concerning those designs $\De$ for which $\L_\De$ is a group. This work fits in to the general programme proposed in \S\ref{s: ext} of this paper to classify those groups that can arise as hole stabilizers of supersimple designs (see Question~\ref{q: classify}).

\subsection{Relation to the literature}

There is a plethora of combinatorial puzzles in recreational mathematics that are `naturally' associated with the generation of a particular finite group, e.g. the already-mentioned 15-puzzle, the Rubik's cube, and many others (see \cite{Mul} for an excellent treatment of this subject).

In recent years various authors have attempted to study group generation \emph{systematically} by referring to variants of such recreational puzzles. In 1974, Wilson generalized the 15-puzzle by defining `moves' on an arbitrary simple graph with $n$ vertices, and he studied the groups generated by the associated permutations of vertices \cite{wilson}; in general the groups obtained contain $A_n$, however in one case, when $n=7$, the corresponding group is isomorphic to $S_5$. A variation of Wilson's construction was studied by Scherphuis \cite{Scherp}, and later by Yang \cite{Yang}.

A different generalization, also using graphs, has recently appeared in \cite{Eos}. The construction described there is a `slide-and-swap game', and it is performed on a cubic graph. Again groups are generated via permutations associated with moves and, in particular, the simple group $\mathrm{PSL}_2(7)$ is obtained via one such construction.

The most successful work in this area is undoubtedly the two papers \cite{Co1, Co} which obtain the notable success of generating a sporadic simple group using permutations obtained via moves on the finite geometry $\mathbb{P}_3$. One naturally wonders whether the work in \cite{Co1, Co} gives yet another example of sporadic behaviour from the sporadic simple groups, or whether it points the way to that most elusive of mathematical goals, a uniform description of all of the sporadic simple groups. This question places the study of such constructions in a wider mathematical context associated with the generation of the finite simple groups (see for instance \cite{Curtis}, and many others).

\subsection*{Acknowledgements}
We would like to thank Tom McCourt for helpful discussions and Brendan McKay for supplying us with a list of all $2$-$(13,4,2)$ designs and hence allowing us to complete the proof of Theorem C. We also thank Andy Chermak for pointing out that $M_{13}$ may be regarded as an objective partial group, which ultimately provided the starting point for this work.

\section{Background}
\label{s:background}

In this section we briefly provide the necessary background definitions for block designs, permutation groups and partial groups.

\subsection{Block designs}
\label{sub:design}

We briefly collect some key notions to do with block designs, and refer the reader to  \cite{BiggsWhite, Handbook} for more information.

Let $t,n,k$ and $\lambda$ be positive integers. A \textit{balanced incomplete block design} $(\Omega,\B)$, or $t$-$(n,k,\lambda)$ design, is a finite set $\Omega$ of size $n$, together with a finite multiset $\B$ of subsets of $\Omega$ each of size $k$ (called \textit{lines}), such that any subset of $\Omega$ of size $t$ is contained in exactly $\lambda$ lines. In particular, then, a $t$-$(n,k,\lambda)$ design is an $k$-hypergraph. A $t$-$(n,k,\lambda)$ design is \emph{simple} if there are no repeated lines, and is called \emph{pliable} if, whenever three points are all contained in two lines, the two lines contain exactly the same points.

In this paper we are mainly interested in simple, pliable $2$-$(n,k,\lambda)$ designs. Observe that, in such a design, any two lines intersect in at most two points. A design with this property is called \emph{supersimple}.

The following special case will be of particular interest to us. The \textit{Boolean quadruple system of order $2^k$} is the design $\De=(\Omega,\B)$ where $\Omega$ is identified with the set of vectors in $\mathbb{F}_2^k$, and $$\B:=\{\{v_1,v_2,v_3,v_4\} \mid v_i \in \Omega \mbox{ and } \sum_{i=1}^4 v_i = \underline{0}\}.$$ It is easy to see that $\De$ is both a $3$-$(2^k,4,1)$ Steiner quadruple system and a $2$-$(2^k,4,2^{k-1}-1)$ design.

We remark that for various small values of the parameters $t,n,k$ and $\lambda$ the set of $t$-$(n,k,\lambda)$ designs has been completely enumerated. We will use this information at various points for the purposes of computer calculation; we refer the reader to \cite[II.1]{Handbook} for more information.

Finally, we caution the reader that although the literature around designs frequently uses the word `block' as a synonym for `line' (hence the name `block design'), we will never do this. For us, the word `block' will always be used with reference to a system of imprimitivity (see the next subsection).

\subsection{Permutation groups}
\label{sub:permutation}

We briefly collect some key notions to do with permutation groups, and refer the reader to \cite{BiggsWhite, Permutation} for more information. Suppose that $G$ is a group acting on a non-empty set $\Omega$. 

The action is called \emph{transitive} if for any $x, y \in \Omega$ there exists $g \in G$ such that $x^g = y$. The action is called \emph{$t$-homogeneous} if the induced action on the set of all subsets of $\Omega$ of size $t$ is transitive. (Contrast this with the stronger notion of a \emph{$t$-transitive} action in which the induced action on the set of all $t$-tuples of distinct elements of $\Omega$ is transitive.)

Suppose that the action of $G$ on $\Omega$ is transitive. A \textit{system of imprimitivity} is a partition of $\Omega$ into $\ell$ subsets $\Delta_1,\Delta_2,\ldots,\Delta_\ell$ each of size $k$ such that $1< k,\ell< n$, and so that for all $i\in\{1,\dots,\ell\}$ and all $g\in G$, there exists $j\in\{1,\dots,\ell\}$ such that
\[
 \Delta_i^g=\Delta_j.
\]
The sets $\Delta_i$ are called \textit{blocks}. We say that $G$ acts \textit{imprimitively} if there exists a system of imprimitivity. If no such set exists then $G$ acts \textit{primitively} on $\Omega$.

Define the \emph{support} of an element $g\in G$, denoted $\supp(g)$, to be the set of points not fixed by a permutation $g$ and denote the cardinality of $\supp(g)$ by $|\supp(g)|$.

\subsection{Partial Groups}\label{sub:part}
We describe objective partial groups.
For a full introduction to the emerging theory of partial groups including basic properties and examples, we direct the reader to \cite{Ch}. We start with some notation. Let $X$ be a set and let $W(X)$ denote the \textit{free monoid} on $X$. An element of $W(X)$ is thus a finite sequence of (or \textit{word} in) the elements of $X$ and multiplication in $W(X)$
consists of concatenation of sequences (denoted $u \circ v$). The \textit{empty word} is the word of length 0, and it is the identity element of the monoid $W(X)$. No distinction is made between the set $X$ and the set of words of length 1. We can now define a partial group:

\begin{Def}\label{partgrp}
Let $\L$ be a non-empty set and let $\W=W(\L)$. A \emph{partial group} is a quadruple $(\L,\D,\Pi,(-)^{-1})$ (often simply denoted $\L$) where:
\begin{itemize}
\item[(a)] $\D$ is a subset of $\W$ with the properties that $\L \subseteq \D$ and $$u \circ v \in \D \Longrightarrow u,v \in \D;$$
\item[(b)] $\Pi:\D \rightarrow \L$ is a \textit{product map} which restricts to the identity on $\L$ and satisfies:
 $$u \circ v \circ w \in \D \Longrightarrow u \circ (\Pi(v)) \circ w \in \D \mbox{ and } \Pi(u \circ v \circ w)=\Pi(u \circ (\Pi(v)) \circ w);$$
\item[(c)] $(-)^{-1}: \L \rightarrow \L$ is an \textit{inversion map} with the following properties:
\begin{itemize}
\item[(i)] $(-)^{-1}$ is involutory and bijective, and induces a map on $\W$ defined by $$(x_1, \ldots, x_n) \mapsto (x_n^{-1},\ldots,x_1^{-1}) \mbox{ and }$$
\item[(ii)] $u \in \D \Longrightarrow u^{-1} \circ u \in \D$ and $\Pi(u^{-1} \circ u)= \Pi(\emptyset)$.
\end{itemize}
\end{itemize}
\end{Def}

Notice that (a) implies that the empty word is also contained in $\D$. We think of $\D$ as a set of words in $\L$ for which products are defined and where associativity holds. Several elementary consequences of the definition are worked out in \cite[Lemmas 2.2-3]{Ch}, but we will not need these. Instead, we present a simple example which shows that every group is a partial group.

\begin{Ex}\label{ex2}
Let $(\L, \cdot)$ be a group. Then $(\L,\D,\Pi,(-)^{-1})$ is a partial group where:
\begin{itemize}
\item[(i)] $\D$ consists of all words in $\L$, i.e.\ $\D=W(\L)$;
\item[(ii)] $\Pi$ is the multivariable product in $\L$;
\item[(iii)] $(-)^{-1}$ sends $g$ to $g^{-1}$.
\end{itemize}
Conversely, any partial group  $(\L,\D,\Pi,(-)^{-1})$ with the property that $\D=W(\L)$ induces a group $(\L,\cdot)$ where $\cdot$ is the binary operation given by restricting $\Pi$ to $\L \times \L.$
\end{Ex}

One may also define partial subgroups of partial groups in a natural way.

\begin{Def}\label{subpartg}
Let $(\L,\D,\Pi,(-)^{-1})$ be a partial group and let $\H$ be a non-empty subset of $\L$. $\H$ is a \emph{partial subgroup} of $\L$ if the following conditions hold:
\begin{itemize}
\item[(a)] $f^{-1} \in \H$ whenever $f \in \H$, and
\item[(b)] $\Pi(w) \in \H$ whenever $w \in W(\H) \cap \D$.
\end{itemize}
If, in addition, $W(\H) \subseteq \D$, $\H$ is a \textit{subgroup} of $\L$ and we write $\H \leq \L$.
\end{Def}

It is immediate from the definition that if $\H$ is a partial subgroup of a partial group $(\L,\D,\Pi,(-)^{-1})$ then $(\H,\D_\H,\Pi,(-)^{-1})$ also has the structure of a partial group where $\D_\H=\D \cap W(\H).$ Thus by Example \ref{ex2}, $\H$ is also a group if $\H$ is a subgroup of $\L.$ 


From now on, we omit the $\Pi$ symbol and simply denote the image under $\Pi$ of the word $w:=(f_1,f_2, \ldots, f_n) \in \D$ by $f_1f_2\cdots f_n.$ When $\L$ is a partial group and $f \in \L$ we define $$\D(f):=\{x \in \L \mid f^{-1}xf \in \D \},$$ i.e.\ $\D(f)$ is the set of elements of $\L$ for which conjugation by $f$ is defined. Further, when $X,Y$ are subgroups of $\L$ define:
\[\begin{aligned}
N_\L(X,Y)&:=\{f \in \L \mid X \subseteq \D(f) \mbox{ and } f^{-1}Xf \leq Y\},   \textrm{ where }\\
 f^{-1}Xf&:=\{f^{-1}xf \mid f\in X\}, \textrm{ for }f \in \L.
\end{aligned}
\]
We can now introduce the concept of an objective partial group (see also \cite[Definition 2.6]{Ch}). This should be thought of as a partial group where all multiplication is `locally' determined.

\begin{Def}\label{obj}
Let $\L$ be a partial group and let $\Delta$ be a set of subgroups of $\L$. Write $\D_\Delta$ for the set of words $w=(f_1,\ldots, f_n) \in W(\L)$ such that:
\begin{equation}\label{objcond}
\mbox{ there is $(X_0,\ldots,X_n) \in W(\Delta)$ such that $(X_{i-1})^{f_i}=X_i$ for all $1 \leq i \leq n$.} 
\end{equation}
Then $(\L,\Delta)$ is an \textit{objective partial group} with object set $\Delta$ if the following conditions hold:
\begin{itemize}
\item[(O1)] $\D=\D_\Delta$;
\item[(O2)] whenever $X,Z \in \Delta$, $Y \leq Z$ and $f \in \L$ is such that $X^f \leq Y \leq Z$, $Y \in \Delta$.
\end{itemize}
\end{Def}

Currently, the main source of examples of objective partial groups are those which arise from fusion systems over finite $p$-groups, and we do not discuss those here (see \cite{Ch} for a full account). The connection with Conway's $M_{13}$ has already been mentioned by Chermak in unpublished work, and we extend and formalise this connection in the next section.

\section{Hole stabilizers}\label{s:puzz}

Let $\Omega$ be a set of size $n$ and let $\De:=(\Omega,\B)$ be a $4$-hypergraph. Recall that $\De$ is \textit{pliable} if whenever three points are all contained in two lines, then the two lines contain exactly the same points. Two points $x,y \in \Omega,$ are \emph{collinear} if there is some line in $\B$ that contains both $x$ and $y$. (Note that any point is collinear with itself.)

Suppose that a distinct pair of elements, $x,y \in \Omega,$ are collinear. Then we define the associated \textit{elementary move} to be the permutation 
$$[x,y]:=(x,y)\prod_{i=1}^\lambda (x_i,y_i) \in \Sym(\Omega),$$
 where $\{x,y,x_i,y_i\}$ is a line for each $1 \leq i \leq \lambda$. Note that, since $\De$ is pliable, this product is well-defined and $[x,y]$ is an involution equal to $[y,x]$.
 
We define the \textit{trivial move} by setting $[a,a]:=\Id_\Omega$ for each $a \in \Omega$, and we define a \textit{move sequence} $$[a_0,a_1,\ldots,a_k]:=[a_0,a_1] \cdot [a_1,a_2] \cdots [a_{k-1},a_k]$$ where $a_{i-1}$ and $a_i$ are collinear elements of $\Omega$ for each $1 \leq i \leq k$. Observe that, since $[a_{i-1},a_i]=[a_i,a_{i-1}]$ for $0 \leq i \leq k$,
\begin{equation}\label{e: inv close}
 [a_0,a_1,\ldots,a_k]^{-1}=[a_k,\ldots,a_1,a_0].
\end{equation}

A move sequence $[a_0,a_1,\ldots,a_k]$ is called \textit{closed} if $a_0=a_k$.  The \textit{puzzle set} $\L_\De$ is the set of all move sequences, that is
\begin{equation}\label{e:fulldef1}
\L_\De:=\{[a_0,a_1,\ldots,a_k] \mid k\in\mathbb{Z}^+; \textrm{$a_{i-1}, a_i\in\Omega$ are collinear for $1 \leq i \leq k$}\}.
\end{equation}
Note that \eqref{e: inv close} implies that $\L_\De$ is closed under inversion.

For each $x \in \Omega$, define the associated \emph{hole stabilizer} to be the set of all closed move sequences which start and end at $x$, that is
$$\pi_x(\De):=\{[a_0,a_1,\ldots,a_k] \in \L \mid a_0=a_k=x\}$$

If $f:=[a_0,\ldots, a_n]$ and $g:=[b_0,\ldots, b_m]$ are permutations in $\L$ we say that the product $f*g$ is \textit{defined} if $a_n=b_0$ and we write $$f*g:=[a_0,\ldots, a_n, b_1 \ldots, b_m ].$$
 It is an easy exercise to confirm that $\pi_x(\De)$ is a group under $*$.\footnote{It is also possible to derive a \textit{signed} version of this construction, mirroring that considered by Conway et al in \cite{Co} for $M_{13}$. However we do not pursue that here.}

Our first result will be useful when we come to study $2$-$(n,4,\lambda)$ designs. Note, in particular, that the supposition of the result holds for any $2$-$(n,4,\lambda)$ design.

\begin{Lem}\label{lem1}
Suppose that $\De$ is a pliable $4$-hypergraph such that any pair of points are collinear. Fix an element $$f:=[a_0,a_1,\ldots, a_n] \in \L_\De.$$ The following statements hold:
\begin{itemize}
\item[(a)] $f=[a_0,a_1,\ldots a_i]\cdot[a_i,a_{i+1},\ldots,a_n]$ for all $1 \leq i \leq n-1$;
\item[(b)] $f=[a_0,a_1,\ldots,a_i,x,a_i,a_{i+1},\ldots,a_n]$ for each  $0 \leq i \leq n$ and $x \in \Omega$;
\item[(c)]  for each $x \in \Omega$, 
\begin{enumerate}
\item[(i)] $\L_\De=\bigcup\limits_{a,b \in \Omega} [a,x]\cdot \pi_x(\De) \cdot [x,b]$; and 
\item[(ii)] if $a,b\in\Omega$ are distinct, then $[a,x] \cdot \pi_x(\De) \cap [b,x] \cdot \pi_x(\De) = \emptyset$.
\end{enumerate}
In particular, $|\L_\De| \geq n\cdot |\pi_x(\De)|;$
\item[(d)] $\pi_x(\De) = \langle [x,a,b,x] \mid a,b \in \Omega \backslash \{x\}\rangle.$

\end{itemize}
\end{Lem}

\begin{proof}
Part (a) is immediate from the definition. To see part (b), note that 
\[
\begin{aligned}[]
[a_0,a_1,\ldots,a_i,x,a_i,a_{i+1},\ldots,a_n]&=[a_0,a_1,\ldots a_i]\cdot[a_i,x]\cdot[x,a_i] \cdot [a_i,a_{i+1},\ldots,a_n] \\
&=[a_0,a_1,\ldots a_i]\cdot [a_i,a_{i+1},\ldots,a_n] \\
&=[a_0,a_1,\ldots a_n]=f.
\end{aligned}
\]
By part (b), for any $x \in \Omega$, $f$ may be written as a product $$[a_0,x]\cdot[x,a_0,\ldots, a_n,x]\cdot[x,a_n],$$ so that $f \in  [a_0,x]\cdot \pi_x(\De) \cdot [x,a_n]$. Conversely each element in any such double coset must lie in $\L$, proving (c)(i). To see (c)(ii) observe that any element in $[a,x] \cdot \pi_x(\De)$ moves the point $a$ to the point $x$, while any element in $[b,x] \cdot \pi_x(\De)$ moves the point $b$ to the point $x$. Thus these two sets have empty intersection as required.


It remains to prove (d). Fix an element $g:=[x,a_1,\ldots,a_{n-1},x] \in \pi_x(\De)$. If $x \in \{a_1,\ldots, a_{n-1}\}$, then $g$ may be written as a product of two elements in $\pi_x(\De)
$, so we are reduced to the case where $x \notin \{a_1,\ldots, a_{n-1}\}$ . By part (b), $$g=[x,a_1,a_2,x,a_2,a_3,x,\ldots, x, a_{n-2},a_{n-1},x]=\prod_{i=1}^{n-2} [x,a_i,a_{i+1},x].$$ This proves (d), and completes the proof of the lemma.
\end{proof}

Let $W(\L_\De)$ be the set of all words in elements of $\L_\De$ and define:
\begin{equation}\label{e:fulldef2}
\begin{aligned}
\D&:=\{(f_1,\ldots, f_n) \in W(\L_\De) \mid f_i * f_{i+1} \mbox{ is defined } \forall \mbox{ } 1 \leq i \leq n-1 \}; \\
\Pi& \textrm{ to be concatenation of move sequences;} \\
(-)^{-1}& \textrm{ to be reversal of move sequences.}
\end{aligned}
\end{equation}

We are ready to prove our first main result, which was stated in \S\ref{sub:main}.


\begin{proof}[Proof of Theorem A]
First we show that $(\L,\D,\Pi,(-)^{-1})$ is a partial group. By construction, Definition \ref{partgrp} (a) holds. Similarly, since $\Pi$ concatenates elements of $\L$ (as in Lemma \ref{lem1} (a)), Definition \ref{partgrp} (b) holds trivially. Finally, by construction, $(-)^{-1}$ is involutory and bijective on $\L$, and by expressing each element of $\D$ as a product of elementary moves it is easy to see that the extension of $(-)^{-1}$ to $\W(\L)$ satisfies (c)(ii) in Definition \ref{partgrp}. This completes the proof that $(\L,\D,\Pi,(-)^{-1})$ is a partial group.

It remains to prove that, given the extra supposition, $(\L_\De,\Delta)$ is objective. For brevity, we regard elements of $\L_\De$ as vectors $\underline{a}:=[a_1,a_2,\ldots,a_n]$ where $a_i \in \Omega$ and $n > 1$.
Write $\Delta_x$ for $\pi_x(\De)$ for each $x \in \Omega$ and observe that $\underline{u}^1\cdot\underline{u}^2 \cdots  \underline{u}^k \in \D$ if and only if $(\Delta_{u_1^i})^{\underline{u}^i}=\Delta_{u_1^{i+1}}$ for all $1 \leq i \leq k-1$ so (O1) holds.

Finally, suppose that $x,y \in \Omega$, $f \in \L$ is such that $\Delta_x \subseteq \D(f)$ and $Y$ is such that $(\Delta_x)^f \leq Y \leq \Delta_y$. Then there exist $k > 0$ and elements $a_1,\ldots,a_k \in \Omega$ such that $f=[x,a_1,\ldots,a_k,y].$ Hence $\Delta_y \subseteq \D(f^{-1})$ and $$|\Delta_x|= |(\Delta_x)^f|=|\Delta_y|=|Y|.$$ In particular, $Y=\Delta_y \in \Delta,$ which proves that (O2) holds. This completes the proof.

\end{proof}

\section{Examples}\label{exsext}
In this section we motivate the main results of this paper by considering some examples. Although we will mainly be interested in those $4$-hypergraphs which arise from designs, our first example is not of this flavour.

\begin{Ex}
Let $n > 2$ and $K_n$ be the complete graph on $n$ vertices labelled by $1,\ldots, n$. Let $\Omega$ be a set of size $2n$ consisting of points $\{x_i,y_i \mid 1 \leq i \leq n\}$ and let $\B$ be the set $\{\{x_i,y_i,x_j,y_j\} \mid ij \in E(K)\}$. It is easy to see that $\De=(\Omega,\B)$ is a pliable $4$-hypergraph and using GAP \cite{GAP}, one can check that for each $\infty \in \Omega$,

$$\pi_\infty(\De) \cong \left\{
	\begin{array}{ll}
		S_2 \wr S_{n-1},  & \mbox{if $n$ is odd}; \\
		Q_n, & \mbox{if $n$ is even,}  
	\end{array}
\right.$$
where $Q_n$ is any index 2 subgroup of $S_2 \wr S_{n-1}$ which does not factor as a direct product.
\end{Ex}

Here is another rather special example, constructed from the unique supersimple $2$-$(10,4,2)$ design (see \cite[II.1.25]{Handbook}).

\begin{Ex}\label{codeex}
Let $\De=(\Omega,\B)$ be the unique supersimple $2$-$(10,4,2)$ design.
Thus $|\B|=15$ and one checks that the following is true:
\begin{equation}\label{21042}
\mbox{\textit{if \{p,q,r,s\} and \{r,s,t,u\} are lines, then \{p,q,t,u\} is a line.} }
\end{equation}

Let $M$ be the incidence matrix for $\De$ viewed as a matrix over $\mathbb{F}_2$. That is, $M$ is a $15 \times 10$ matrix where rows are indexed by lines and columns are indexed by points and where $$m_{ij} = \left\{
	\begin{array}{ll}
		1,  & \mbox{if $j$ is a point in $i$}; \\
		0, & \mbox{otherwise. } 
	\end{array}
\right.$$

Let $C$ be the linear code (vector space) spanned by the rows of $M$. Also, for each $p \in \Omega$, define: $$C_p:=\{\underline{c} \in C \mid c_p=0\}.$$ Then, using (\ref{21042}), one easily verifies that for each $p,q \in \Omega$ the element $[p,q]$ sends codewords in $C_p$ to codewords in $C_q$ by permuting the coordinates entrywise. In particular $\pi_p(\De)$ acts as a group of automorphisms of $C_p$, and a GAP computation \cite{GAP} reveals that in fact $$\pi_\infty(\De) \cong S_3 \wr S_2 \cong \Aut(C_p)$$ for each $p \in \Omega$. Indeed, there is more: one can verify that $\L_\De$ consists of 720 permutations which together form a primitive subgroup of $S_{10}$ isomorphic with $S_6$.

\end{Ex}

We remark that the construction above is inspired by that found in \cite[Section 3]{Co} where the authors construct the ternary Golay code from the incidence matrix for $\mathbb{P}_3.$ 



\begin{table}
\begin{center}
\begin{tabular}{|c| c| c| c| c| c|} 
\hline
$n$&$\lambda$&$\Aut(\De)$& $\pi_\infty(\De)$&Action&$\L_\De$\\
\hline
8&3&$\mathrm{AGL_3}(2)$&1&trivial&$\L_\De=(C_2)^3$\\
\hline
9&3&$\mathrm{AGL}_1(9)$&$A_4 \wr C_2$&transitive&$|\L_\De|>9\cdot |\pi_\infty(\De)|$\\
\hline
10&2&$S_6$&$S_3 \wr S_2$& primitive&$\L_\De=S_6$\\
\hline
13&1&$\mathrm{PSL}_3(3)$&$M_{12}$&primitive&$|\L_\De|>13\cdot |\pi_\infty(\De)|$\\
\hline
16&6&$\mathrm{AGL}_2(4)$&$(S_3)^5$&intransitive&$|\L_\De|>16\cdot |\pi_\infty(\De)|$\\
\hline
16&3& $2^4.S_6$&$S_6$&primitive&$\L_\De=2^4.S_6$\\
\hline
16&7&$\mathrm{AGL}_4(2)$&1&trivial&$\L_\De=(C_2)^4$\\
\hline
17&6&$\mathrm{AGL}_1(17)$&$S_8 \times S_8$&intransitive&$|\L_\De|>17\cdot |\pi_\infty(\De)|$\\
\hline
28&5&$\mathrm{Sp}_6(2)$&$\mathrm{PSp}_4(3):2$&primitive&$\L_\De=\mathrm{Sp}_6(2)$\\
\hline
32&15&$\mathrm{AGL}_5(2)$&$1$&trivial&$\L_\De=(C_2)^{5}$\\
\hline
36&9&$\mathrm{Sp}_6(2)$&$S_8$&primitive&$\L_\De=\mathrm{Sp}_6(2)$\\
\hline
49&18&$-$&$S_{24}\times S_{24}$&intransitive&$|\L_\De|>49\cdot |\pi_\infty(\De)|$\\
\hline
\end{tabular}
\vspace{0.1in}
\end{center}
\caption{$2$-$(n,4,\lambda)$ designs $\De$ for $n\leq 50$, $\Aut(\De)$ (if known), the corresponding hole stabilizer $\pi_\infty(\De)$ (and its action), and a description of the puzzle set $\L_\De$.}
\end{table}

Now let $X$ be a $2$-homogeneous primitive permutation group of degree $n$, so that the orbit $\mathcal{O}$ of a block of size 4 is necessarily a $2$-$(n,4,\lambda)$ design $\De$ for some $\lambda > 0$. Using the GAP library of primitive groups, we computed $\pi_\infty(\De)$ for supersimple designs $\De$ which arise in this way for all $n \leq 50.$ Table 1 provides a list of those designs $\De$ for which $\pi_\infty(\De)$ is not symmetric or alternating. All groups are described using the \textsc{atlas} notation \cite{atlas}. 

\begin{remark}\label{r. m13}
Note that an equality of the form $``\L_\De=''$ in the last column of Table 1 does not indicate that, as a partial group, $\L_\De$ also has the structure of a group. Instead we are asserting that \emph{the underlying set} in $\L_\De$ coincides with the set of elements in some subgroup of $\Sym(\Omega)$. 
\end{remark}

\begin{remark}
We noted in the introduction that $\L_\De \supseteq M_{13}$, where the definition of $M_{13}$ is as follows \cite{Co}: $$M_{13}:=\{[\infty,a_1,\ldots,a_k] \mid k \in \mathbb{Z}^+; a_i \in \mathbb{P}_3 \mbox{ for } 1 \leq i \leq k \}.$$ In fact $M_{13}$ is a \emph{proper} subset of $\L_\De$. To see this, recall that in \cite{Co1}, Conway proves that $|M_{13}|=13 \cdot |M_{12}|$. On the other hand  Lemma \ref{lem1} (c)(ii) asserts that $|\L_\De| \gneq 13\cdot |\pi_\infty(\De)|$; moreover, by the same result, equality would imply that for any pair of points $x,y$ with $\{x,y,\infty\}$ not contained in a line, the permutation $[x,y] \in \pi_\infty(\mathbb{P}_3) \cong M_{12}$. However $\supp([x,y])=4$, while every element of $M_{12}$ has support of size at least $8$ (see Theorem \ref{t: bounds}), a contradiction.
\end{remark}

Observe that, for each entry in Table 1 such that $\pi_\infty(\De)$ is the trivial group, $(n,\lambda)$ is of the form $(2^k,2^{k-1}-1)$ for some $k > 0$. Indeed this is necessarily the case as we prove in the next section.

\section{Trivial hole stabilizers}
\label{sec:theorem b}
In this section we will prove Theorem B. Let $\De:=(\Omega,\B)$, where $\Omega=\mathbb{F}_2^k$, be a Boolean quadruple system. One observes that  $\De$ is simultaneously a $3$-$(2^n,4,1)$ design and a $2$-$(2^n,4,2^{n-1}-1)$ design. When $\De$ is a pliable $4$-hypergraph and $a,b \in \Omega$, write $\bar{a,b}$ for the set $\{a,b\}$ union with the set of points in $\Omega$ contained in some line with $\{a,b\}$.

\begin{Lem}\label{l:triv1}
Let $\De=(\Omega,\B)$ be the Boolean quadruple system of order $2^k$, where $\Omega=\mathbb{F}_2^k$. Then $\L_{\De}$ is the image of $\Omega$ under the the regular action. Consequently,  $$\pi_\infty(\De)=\stab_{\L_\De}(\infty)=1,$$ for each $\infty \in \Omega.$
\end{Lem}

\begin{proof}
Let $$\begin{CD}
\rho: \Omega  @>>>  \Sym(\Omega) 
\end{CD}$$ be the regular action of $\Omega$ given (for each  $\omega \in \Omega$) by $\omega \rho := \sigma_\omega$, where $v\sigma_\omega = v+\omega$ for each $v \in \Omega.$ It suffices to observe that $\sigma_\omega=[a,b]$ for any $a,b \in \Omega$ satisfying $a+b=\omega$. Indeed, $a\sigma_\omega=b$ and for each $c \in \Omega \backslash \{a,b\}$, $c\sigma_\omega=c+\omega=c+a+b,$ as needed. This completes the proof.
\end{proof}




We can now prove Theorem B.

\begin{proof}[Proof of Theorem B]
Let $\De=(\Omega,\B)$ be a pliable $4$-hypergraph with trivial hole stabilizer. For each $\infty,a,b \in \Omega$, it is easy to see that $\infty \notin \overline{a,b}$ implies that the permutation $[\infty,a,b,\infty]$ sends $a$ to $b$, a contradiction. Hence each triple of elements is collinear, and since $\De$ is pliable, $\De$ is a $3$-$(n,4,1)$ design.

We now associate to $\Omega$ the following binary operation $*$ given by:
\begin{itemize}
\item[(i)] $\infty*a=a*\infty=a$ for all $a\in\Omega$;
\item[(ii)] $a*a=\infty$ for all $a\in\Omega$;
\item[(iii)] $a*b:=c$ for all $a,b\in\Omega\backslash\{\infty\}$ with $a\neq b$ where $\{a,b,c,\infty\}\in\mathcal{B}$.
\end{itemize}
It is a consequence of this definition that $a*b=b*a$ for all $a,b \in \Omega.$  We claim that $*$ is associative, that is, $(a*b)*c=a*(b*c)$ for all $a,b,c\in\Omega.$
The case $a=b=c$ follows from (i) and (ii).  If $a\neq b$ then $(a*b)*a=a*(a*b)=a*(b*a)$,
and it follows from (iii) that $(a*a)*b=\infty*b=b=a*d=a*(a*b)$, where $\{\infty,a,b,d\} \in \B$.
It remains to consider the case where $a,b,c$ are pairwise distinct. Let $\{a,b,c,x\}$ be the unique line containing $\{a,b,c\}$. If $x=\infty$, then $$(a*b)*c=c*c=\infty=a*a=a*(b*c).$$	
Otherwise, let $a*b=s$, $b*c=t$, $s*c=p$
and $a*t=q$, so that it suffices to show $p=q$.  

Our assumptions imply that 
\begin{align*}
[\infty,a]=(\infty,a)(b,s)(t,q)\Pi_1, \,\,\,\,\,\,  & [a,b]=(a,b)(\infty,s)(c,x)\Pi_2,\\
[b,c]=(b,c)(\infty,t)(a,x)\Pi_3, \,\,\,\,\,\,  & [c,\infty]=(\infty,c)(s,p)(b,t)\Pi_4,
 \end{align*}
where $\Pi_i$ is a product of transpositions for $1 \leq i \leq 4$. Since
$$\tau:=[\infty,a,b,c,\infty]=[\infty,a][a,b][b,c][c,\infty]=\Id_\Omega,$$
we have $q=t^{[\infty,a]}=t^{[c,\infty][b,c][a,b]}=x$, and hence $$q=s^{[\infty,a][a,b][b,c]}=s^{[c,\infty]}=p,$$ proving the claim.

We conclude that $(\Omega,*)$ is an abelian group with identity element $\infty$.  Moreover, from (ii) above, each non-identity element has order $2$ and hence we may identify $(\Omega,*)$ with $(\mathbb{F}_2^k,+)$ for some $k>0$. It thus remains to check that $a*b*c*d=\infty$ whenever $\{a,b,c,d\}$ is a line. This clearly holds if $\infty \in  \{a,b,c,d\}$ so we may assume this not the case. Let $a*b=x$ and $x*c=w$, so that
\[
[\infty,b]=(\infty,b)(x,a)\Pi_1, \,\,
[b,c]=(b,c)(a,d)\Pi_2, \textrm{  and  }
[c,\infty]=(c,\infty)(x,w)\Pi_3,
\]
where $\Pi_i$ is a product of transpositions for $1 \leq i \leq 3$. Since $[\infty,b,c,\infty]=1,$ $d=x^{[\infty,b][b,c]}=x^{[c,\infty]}=w$. Since $*$ is associative, this proves that $a*b*c=d$, and hence $a*b*c*d=\infty$, as required.
\end{proof}








\section{The behaviour of \texorpdfstring{$n$ and $\lambda$}{n and lambda}}\label{s: lambda12}
Consider $\De=(\Omega,\B)$, a supersimple $2$-$(n,4,\lambda)$ design. In this section we prove a number of results of similar ilk: we assume that $n$ satisfies some inequality with respect to $\lambda$ and we draw conclusions as to the structure of $\pi_\infty(\De)$. These results will be applied in \S\ref{s: thmc} in our classification of the groups $\pi_\infty(\De)$ associated to simple $2$-$(n,4,\lambda)$ designs with $\lambda\leq 2$.

In what follows, $\De$ is a fixed supersimple $2$-$(n,4,\lambda)$ design and $G:=\pi_\infty(\De)$ is the associated hole stabilizer. Note that no two lines of $\De$ intersect in more than two points.  As in the previous section, for $a,b \in \Omega$, we write $\bar{a,b}$ for the set consisting of points $a,b$ and the $2\lambda$ points in $\Omega$ that are contained in some line with $\{a,b\}$. In particular, $|\bar{a,b}|=2\lambda+2$.

We begin with a criterion for transitivity.

\begin{Lem}\label{transg}
$G$ is transitive for all $n > 4\lambda+1$.
\end{Lem}

\begin{proof}
Suppose that $n > 4\lambda+1$ and fix some $a \in \Omega \backslash \{\infty\}$. For each $b \neq a$ such that $\infty \notin \bar{a,b}$, the element $[\infty,a,b,\infty]$ maps $a$ to $b$. Since $\{a,\infty\}$ is a subset of $\lambda$ lines, this implies that $|a^G| \geq n-1-2\lambda$. Since $a$ was arbitrary, if $\pi_\infty(\De)$ is not transitive on $\Omega \backslash \{\infty\},$ then for some $b \notin a^G$, $$2(n-1-2\lambda) \leq |a^G|+|b^G|  \leq n-1,$$ which implies that $n \leq 4\lambda+1$, as needed.
\end{proof}




\begin{Lem}\label{l: prim}
Let $n > 4\lambda+1$ and suppose that $G$ preserves a system of imprimitivity with $\ell$ blocks each of size $k$ (so that $n-1=k\ell$). Then at least one of the following holds:
\begin{enumerate}
\item[(i)] if $a,c\in \Omega$ lie in the same block of imprimitivity, then $\infty\in\overline{a,c}$;
\item[(ii)] $n\leq \frac{6\ell}{\ell-1}\lambda+1$.
 \end{enumerate}
\end{Lem}
\begin{proof}
Suppose that (i) does not hold so that there exist $a,c \in \Omega$ which lie in the same block of imprimitivity with $\infty\not\in\overline{a,c}$. We consider the orbit of $a$ under $\stab_G(c)$. Observe that, for $b\in \Omega\backslash\{a,c,\infty\}$, the element $g=[\infty, a,b,\infty]$ satisfies $a^g=b$ and $c^g=c$ provided $$b\not\in\overline{a,c} \cup \overline{\infty,a} \cup \overline{\infty,c}.$$ This implies that $|a^{\stab_G(c)}| \geq n-2-6\lambda$, so that $k \geq n-1-6\lambda$. In particular,
 \[
  n-1-6\lambda \leq \frac{n-1}{\ell}
 \]
and the result follows.
\end{proof}

Note, in particular, that if (ii) holds in Lemma \ref{l: prim} then $n\leq 12\lambda+1$. The result gives significantly better information, though. For instance if $9\lambda+1<n\leq 12\lambda+1$ and $G$ is imprimitive, then the only possible systems of imprimitivity must have exactly 2 blocks, i.e.\ $G\leq S_{\frac{n-1}{2}}\wr S_2$.

It will be useful to consider the situation described in Lemma~\ref{l: prim} (i) in a little more detail. The next couple of results give restrictions on $k$ in this situation. 

\begin{Lem}\label{l: bound k above} Suppose that $G$ is transitive and preserves a system of imprimitivity with $\ell$ blocks each of size $k$. Suppose, moreover, that if $a,c\in \Omega$ lie in the same block of imprimitivity $\Delta$, then $\infty\in\overline{a,c}$. Then 
\[k \leq \left\{
	\begin{array}{ll}
		2\lambda-1  & \mbox{if } \lambda\equiv2\pmod3; \\
		2\lambda+1 & \mbox{otherwise. } 
	\end{array}
\right.\]
Furthermore, if $k> 2\lambda-1$, then $k=2\lambda+1$. In this case, let $\Lambda$ be the set of lines that contain two points in $\Delta$ and the point $\infty$. If we remove the point $\infty$ from every element of $\Lambda$, then the pair $(\Delta,\Lambda)$ is a $2$-$(k,3,1)$ design.
\end{Lem}


\begin{proof} Let $a$ be an element of $\Omega\backslash\{\infty\}$ and let $\Delta$ be the block of imprimitivity containing $a$. There are exactly $2\lambda$ elements $c\in\Omega\backslash\{\infty\}$ such that $\infty\in\overline{a,c}$. Thus, by supposition, $k\leq 2\lambda+1$, and this yields the result for $\lambda=1$.


Now suppose that $\lambda\geq2$ and that $k=2\lambda+1$, i.e.\ $\Delta$ is precisely the set of points $c$ such that $\infty\in\overline{a,c}$. Let $c\in \Delta\backslash\{a\}$, and suppose that there exists  $d\in\overline{\infty,c}$ with $d\not\in\Delta$. Consider the block of imprimitivity $\Delta'$ containing $d$. It cannot contain $c$, since $c\in\Delta$, thus it cannot contain every point $e$ such that $\infty\in\overline{d,e}$. Thus $|\Delta'|<2\lambda+1$, a contradiction.

Thus we may assume that, for every $c\in\Delta$, and $d \in \overline{\infty,c}$, we have $d\in\Delta$. For any $c\in\Delta$ there are $\lambda$ lines containing $\{c,\infty\}$ and so there are $2\lambda$ points on these lines other than $c$ and $\infty$. Thus, given any $c\in\Delta$, the set of $\lambda$ lines connecting $c$ to $\infty$ each contain exactly three points from $\Delta$. What is more, by supersimplicity, no point in $\Delta\backslash\{c\}$ occurs more than once in these $2\lambda$ lines.

We conclude that any two points in $\Delta$ lie on a line that contains $\infty$ and three points from $\Delta$. In addition, by supersimplicity, two points in $\Delta$ are connected by no more than one line that includes $\infty$, thus they are connected by exactly one line and we have a $2$-$(k,3,1)$ design as required. Now we observe that a $2$-$(k,3,1)$ design only exists when $k=2\lambda+1\equiv 1$ or $3\pmod 6$ \cite{kirkman}, i.e.\ when $\lambda\not\equiv2\pmod 3$.




Suppose, finally, that $\lambda\geq 2$ and $k=2\lambda$, i.e.\ $\Delta$ contains all but one of the set of points $c$ such that $\infty\in\overline{a,c}$. Let $d$ be this remaining point and consider the block of imprimitivity $\Delta'$ containing $d$. There is a point $b$ such that $\{\infty, a,b,d\}$ is a line and, by assumption, both $a$ and $b$ lie in $\Delta$. This implies that $\Delta'$ is missing at least two points $e$ such that $\infty\in\overline{d,e}$. Thus $|\Delta'|<2\lambda$, a contradiction. This completes the proof.
\end{proof}

\begin{Lem}\label{l: bound k below}
Suppose that $G$ is transitive and preserves a system of imprimitivity with $\ell$ blocks each of size $k$. Suppose, moreover, that if $a,c\in \Omega$ lie in the same block of imprimitivity, then $\infty\in\overline{a,c}$. If $\lambda=2$ then $k=3$.
\end{Lem}
\begin{proof}
By Lemma \ref{l: bound k above} we may suppose, for a contradiction, that $k=2$. Let $\{a_1, a_2\}$ be a block of imprimitivity. By assumption $\infty\in\overline{a_1, a_2}$, and we write $a_3$ for the `other' point in the line containing ${\infty, a_1, a_2}$. Let $\{a_3, a_4\}$ be a block of imprimitivity. Again $\infty\in\overline{a_3,a_4}$, and we write $a_5$ for the `other' point in the line containing $\{a_3, a_4\}$. Finally let $\{a_5,a_6\}$ be a block of imprimitivity. We give labels to various other points in $\De$ so that we have the following set of lines:
\[ \{\infty,a_1 ,a_2 ,a_3 \}, \{\infty, a_3 ,a_4 ,a_5 \}, \{ \infty,a_5 ,a_6 ,v \}, \{a_3,a_5 ,x ,y \}. \]
By construction the points $a_1, a_2,\dots,a_6$ are all distinct; note, too, that $a_5$ is distinct from the points $v,x$ and $y$.

Now one can easily check that
\[
 g:=[\infty,a_3,a_5,\infty]=(a_1,a_2)(x,y)(a_6,v).
\]
Since $g$ fixes $a_5$ it must fix $a_6$ so that $\{x,y\}=\{a_6,v\}$. This contradicts supersimplicity and we are done.
\end{proof}

\begin{Lem}\label{l: prim 2}  Let $n> 12\lambda+1$ and suppose that $G$ preserves a system of imprimitivity with $\ell$ blocks each of size $k$. If $\{\infty,a,b,c\}$ is a line, then $a,b,c$ are not all contained in the same block of imprimitivity.
\end{Lem}

\begin{proof}
Fix a line $\alpha=\{\infty,a,b,c\}$.  
If $k=2$ the result is clear, so suppose that $k\geq 3$ and that $a,b,c\subseteq\Delta$ for
some block of imprimitivity $\Delta$. Define
\begin{equation}\label{setsoflines}
  X:= \overline{\infty,c} \cup \overline{\infty, a} \cup \overline{\infty, b} \cup \overline{a,c} \cup \overline{b,c}
 \end{equation}
and observe that $|X| \leq 10\lambda-6$. Thus since $n> 12\lambda+1$, $|\Omega \backslash X| \neq 0$ and we may choose $e \in \Omega \backslash X$.  One easily checks that the permutation $\tau=[\infty, c,e,\infty]$ interchanges $a$ and $b$, as well as $c$ and $e$.
Thus, $\Delta^\tau=\Delta$, and so $e\in\Delta$.  Since $e \in \Omega \backslash X$ was arbitrary and $|X| \leq 10\lambda-6$, this shows that \begin{equation}\label{eq}k\geq n-10\lambda+9>2\lambda+10.\end{equation} (Observe that $a,b,c \in \Delta$). 
However, it follows from Lemma~\ref{l: prim} and Lemma~\ref{l: bound k above} that $k\leq 2\lambda+1$, a contradiction.    
\end{proof}
 
\begin{Lem}\label{imprimblock} Let $n > 12\lambda+1$ and suppose that $G$ preserves a system of imprimitivity with $\ell$ blocks each of size $k$.  Then $2\leq k\leq\lambda$.
\end{Lem}

\begin{proof}
Consider a block of imprimitivity $\Delta:=\{a_1,\ldots,a_k\}$.  
Lemmas~\ref{l: prim} and~\ref{l: prim 2} imply that there exists a line $\alpha=\{\infty,a_1,a_2,b_1\}$ for some $b_1\notin\Delta$.  Let $\Delta'$ be the block of imprimitivity that contains $b_1$, that is, $\Delta'=\{b_1,\ldots,b_k\}$.
Again, by Lemmas \ref{l: prim} and \ref{l: prim 2}, for $i=1,\ldots,k-1$   there exists a line $\beta_i=\{\infty,b_1,b_{i+1},c_i\}$ for some $c_i\notin\Delta'$. 
As the design is supersimple, it follows that $\alpha,\beta_1,\ldots,\beta_{k-1}$ are $k$ pairwise distinct lines each containing $\{\infty,b_1\}$.  Hence $k\leq\lambda$.
\end{proof}

\begin{Cor}\label{c: prim}
If $\lambda=1$ or $2$ then either $G$ is primitive or $n\leq 13$ or $25$ respectively. 
\end{Cor}

\begin{proof}

Suppose that $G$ is imprimitive and $n$ does not satisfy the given inequality. In particular $n>12\lambda+1$. We apply Lemma~\ref{l: prim} and observe that option (i) of that lemma must hold. Now Lemmas \ref{imprimblock} and \ref{l: bound k below} apply and we obtain a contradiction.

\end{proof}

\section{Proof of Theorem C}\label{s: thmc}

In this section we will apply Corollary \ref{c: prim} in order to classify all hole stabilizers that arise from supersimple designs $\De$ with $\lambda=1,2$. 

In what follows, for a permutation group $H$, we write\footnote{Given a permutation group $H$, the quantity $\mu(H)$ is referred to in the literature by a number of different names, including the \emph{class} or the \emph{minimal degree} of $H$.}
\[
 \mu(H):=\min\{|\supp(h)| \mid h\in H\backslash\{1\}\}.
\]

Our proof of Theorem~C will exploit some classical results that deduce an upper bound on $d$ from an upper bound on $\mu(H)$. The statement we need is an easy consequence of work of Jordan (see \cite[Theorem 13.9]{wielandt}), along with results of Manning (see \cite{manning, manning1}):

\begin{Thm}\label{t: bounds}
Let $H$ be a primitive permutation group of degree $d$ that does not contain $A_d$.
\begin{itemize}
\item[(i)] If $\mu(H) \leq 6$ then $d \leq 10$.
\item[(ii)] If $\mu(H) \leq 8$ then $d \leq 16$.
\end{itemize}
\end{Thm}



\begin{Lem}\label{l: support imprimitive}
 Suppose that a permutation group $H$ preserves a non-trivial system of imprimitivity with $\ell$ blocks each of size $k$. Then any subset of $H$ that generates $H$ must contain an element of support at least $2k$.
\end{Lem}
\begin{proof}
Consider the action of $H$ on the set of $\ell$ blocks. The kernel of this action is a normal subgroup, $N$, of $H$. Clearly $N$ is intransitive on $\Omega$, thus $N$ is a proper subgroup of $H$ and there exists an element $h\in H\backslash N$. Since the element $h$ has support at least $2$ in the action on the $\ell$ blocks, it has support at least $2k$ in the action on $\Omega$, and we are done.
\end{proof}

The following lemma is all that we shall need to prove Theorem~C in many cases.

\begin{Lem}\label{l: mindeg}
Let $\De=(\Omega,\B)$ be a supersimple $2$-$(n,4,\lambda)$ design.
\begin{enumerate}
 \item The group $\pi_\infty(\De)$ contains a set $A$ such that $\langle A\rangle=\pi_\infty(\De)$, and if $g\in A$, then $|\supp(g)|\leq 6\lambda+2$.
 \item If $\lambda$ is even, then $\mu(\pi_\infty(\De)) \leq 6(\lambda-1)$.
\end{enumerate}
\end{Lem}

\begin{proof}
Lemma \ref{lem1} (d) implies that $\pi_\infty(\De)$ is generated by elements of the form ${\tau:=[\infty,a,b,\infty]}$ where $a,b$ are distinct elements of $\Omega \backslash \{\infty\}$. 
If $\infty \in \bar{a,b}$ then $\tau=1$ when $\lambda=1$ and $|\supp(\tau)| \leq 6(\lambda-1)$ when $\lambda > 1$. If $\infty \notin \bar{a,b}$ then $|\supp(\tau)| \leq 6\lambda+2$ for all $\lambda \geq 1$. This proves (1).

For (2) we note that, when $\lambda$ is even, $\tau$ is the product of an odd number of transpositions, and hence cannot be trivial. In particular, if $\infty \in \bar{a,b}$, $\tau$ is a non-trivial element of $\pi_\infty(\De)$ with $|\supp(\tau)|\leq 6(\lambda-1)$ as required.
\end{proof}

\begin{Thm}\label{t: weak C}
Let $n \geq 7$ and $\De=(\Omega,\B)$ be a supersimple $2$-$(n,4,\lambda)$ design with $\lambda\leq 2$. Then either Theorem~C holds or else $\lambda=2$, $13 < n \leq 25$ and $\pi_\infty(\De)$ acts transitively and imprimitively on $\Omega \backslash \{\infty\}$.
\end{Thm}
\begin{proof}
First note that, for $\lambda=1$ or $2$, a supersimple $2$-$(n,4,\lambda)$ design exists 
only if $n \equiv 1,4 ~(\mbox{mod } \frac{12}{\lambda})$. This follows quickly from the observations that
$$|\B|=\frac{n\cdot(n-1)}{4 \cdot 3} \lambda\in \mathbb{Z} \mbox{ and } r=\frac{(n-1)\lambda}{3} \in \mathbb{Z}$$ where $r$ is the number of lines containing a given point.
\newline
\newline
Assume that $\lambda=1$. By \cite[II.1.26]{Handbook}, when $n=13$ there is a unique supersimple $2$-$(13,4,1)$ design $\De$ determined by the lines in the projective plane $\mathbb{P}_3$. In this case $\pi_\infty(\De) \cong M_{12}$ by \cite[Theorem 3.5]{Co}. When $n=16$ there is a unique supersimple $2$-$(16,4,1)$ design $\De$ (see \cite[II.1.31]{Handbook}) and using GAP \cite{GAP}, one easily verifies that $\pi_\infty(\De) \cong A_{15}$ in this case. When $n > 16$, $\pi_\infty(\De)$ is primitive by Corollary \ref{c: prim}, now Lemma \ref{l: mindeg} and Theorem \ref{t: bounds} imply that $\pi_\infty(\De)\geq A_{n-1}$. Finally Lemma \ref{lem1} (d) implies that $\pi_\infty(\De)$ is generated by a set of even permutations, and we are done.

Now assume that $\lambda=2$. If $n \leq 13$, then there is a unique $2$-$(7,4,2)$ design ($\mathbb{P}_2$, the projective plane of order $2$) with $\pi_\infty(\De)=S_6$ and a unique supersimple $2$-$(10,4,2)$ design (see \cite[II.1.25]{Handbook}) with $\pi_\infty(\De)=S_{3}\wr C_2$ (see Table 1). In both of these cases, and also when $n=13$ (when there are 2461 designs to consider (see entry 23 in \cite[II.1.35]{Handbook})), we calculate $\pi_\infty(\De)$ using GAP \cite{GAP}.  For $n=13$, the result is always $S_{12}$, and so Theorem~C holds in this case.

If $n > 25$, then $\pi_\infty(\De)$ is primitive by Corollary \ref{c: prim} and, as before, Lemma \ref{l: mindeg} and Theorem \ref{t: bounds} imply that $\pi_\infty(\De)\geq A_{n-1}$. Then, since any element $[\infty, a,b,\infty]$ (where $a,b$ are distinct elements of $\Omega \backslash \{\infty\}$) is obviously odd, Theorem~C holds in this case.

Finally, when $13<n\leq 25$ we observe that, by Lemma~\ref{transg}, $\pi_\infty(\De)$ is transitive. If $\pi_\infty(\De)$ is primitive, then Lemma \ref{l: mindeg} and Theorem~\ref{t: bounds} imply that $\pi_\infty(\De)$ contains $A_{n-1}$ and, since elements $[\infty, a,b,\infty]$ are odd as noted above, we conclude that $\pi_\infty(\De)\cong S_{n-1}$ and, once, again, Theorem~C holds and we are done. 
\end{proof}

\subsection{Proof of Theorem C}

To finish the proof of Theorem C we must deal with the leftover case in Theorem~\ref{t: weak C}. 

Thus our suppositions for this subsection are as follows: $G$ is the hole stabilizer $\pi_\infty(\De)$ where $\De$ is a supersimple $2$-$(n,4,\lambda)$ design. This means, of course, that $G$ is a subgroup of $S_{n-1}$ and we assume, throughout, that $G$ is transitive and imprimitive; in particular $G$ preserves a non-trivial system of imprimitivity with $\ell$ blocks each of size $k$ (so that $n-1=k\ell$ with $1<k,\ell< n-1$). 

Our aim is to show a contradiction. Note that we do not make any general suppositions about the size of $n$, although some of our results will require a lower bound. Our first result is the only one that applies for any $\lambda$.

\begin{Lem}\label{l: cool}
$k\leq 3\lambda+1$.
\end{Lem}
\begin{proof}
Lemma~\ref{l: mindeg} implies that $G$ is generated by a set all of whose elements have support at most $6\lambda+2$. Now Lemma~\ref{l: support imprimitive} yields the result.
\end{proof}

 From here on we assume that $\lambda=2$.
 
\begin{Lem}\label{l: k not 3}
$k \neq 3$.
\end{Lem}
\begin{proof}
Suppose that $k=3$. If every pair of points $a_1$ and $a_2$ with $\infty\in\overline{a_1,a_2}$ lie in the same block of imprimitivity then $k > 3$, so we conclude that there is some line $\{a,a_1,a_2,\infty\}$ such that $a_1$ and $a_2$ do not lie in the same block of imprimitivity. Suppose that we have the following lines:
\[
 \{a,b,c,\infty\}, \{b,b_1,b_2,\infty\}, \{c, c_1,c_2,\infty\}, \{d_1,b,a,d_2\}.
\]
We do not assume that all of the listed points are distinct. Now observe that
\[
 g:=[\infty, a,b,\infty]=(a_1,a_2)(d_1,d_2)(b_1,b_2).
\]
If $|\supp(g)|<6$, then $g$ fixes all blocks and we conclude that $a_1$ and $a_2$ lie in the same block, a contradiction. Thus $|\supp(g)|=6$, all of the listed points are distinct and, by labelling appropriately, we have two blocks of imprimitivity: $\{a_1,b_1,d_1\}$ and $\{a_2,b_2,d_2\}$.

We know then, that $b_1$ and $b_2$ do not lie in the same block of imprimitivity, and we can run the same argument with respect to the element
\[
 h:=[\infty, b,c,\infty]=(b_1,b_2)(e_1,e_2)(c_1,c_2),
\]
where $\{b,c,e_1,e_2\}$ is a line. We have two blocks of imprimitivity, as before: $\{b_1,c_1,e_1\}$ and $\{b_2,c_2,e_2\}$. Since these blocks intersect non-trivially with the previous, they must coincide, and by repeating the argument with the element $[\infty,a,c,\infty]$ we end up with a configuration of 10 distinct points - $a,b,c,a_1,b_1,c_1,a_2,b_2,c_2,\infty$; 7 lines - 
\[
 \begin{aligned}
  \{a,b,c,\infty\}, \{a,a_1,a_2,\infty\}, \,\, & \{b,b_1,b_2,\infty\}, \,\,  \{c, c_1,c_2,\infty\}, \\
  \{b,c,a_1,a_2\}, \{a,b,c_1,c_2\}, \, \, & \{a,c,b_1,b_2\};
 \end{aligned}
\]
and 2 blocks of imprimitivity: $\Delta_1:=\{a_1,b_1,c_1\}$ and $\Delta_2:=\{a_2,b_2,c_2\}$. Now consider the element
\[
 f:=[\infty, a,c_1,\infty] = (a,c_1)\big((b,c)(a_1,a_2)(b,c_2)(r,s)(c_2,c)(t,u)\big)
\]
where $\{a,c_1,r,s\}$ and $\{c_1,\infty,t,u\}$ are lines. Since $a$ and $c_1$ are interchanged, $f$ must move $\Delta_1$. 
Moreover, $b\notin\{r,s,t,u\}$ (by supersimplicity, and because $\lambda=2$), so $c_2^f=b$ and $f$ moves $\Delta_2$ also.
Since $a,b \notin \{a_1,b_1,c_1,a_2,b_2,c_2\}$, $f$ cannot move $\Delta_1$ to $\Delta_2$, and vice versa, so $|\supp(f)|\geq 12$.
However, it is easy to deduce that $a_1,a_2,b_1,b_2 \in \{r,s,t,u\}$, which implies that $|\supp(f)|\leq 9$, a contradiction.  
\end{proof}

 Next we introduce the notion of an \emph{$\infty$-triangle}. This is a configuration of six points $a,b,c,a_1,b_1,c_1\in \De$ such that
\[
 \{a,b_1,c,\infty\}, \{a_1,b,c,\infty\} \textrm{ and } \{a,b,c_1,\infty\}
\]
are lines. In other words, an $\infty$-triangle is  a triangle of lines, with vertices the points $a$, $b$ and $c$, such that the three lines all contain the point $\infty$. Our next result asserts that for sufficiently large $n$ and $k$, such a configuration must exist within the design.

\begin{Lem}\label{l: no infinity triangle}
If $n\geq 10$ then either $k=2$ or $\De$ contains an $\infty$-triangle.
\end{Lem}
\begin{proof}
By Lemma \ref{l: k not 3} we may assume that $k\geq 4$ and that $\De$ contains no infinity triangle. Suppose that we have nine points $a,a_1,a_2,b,b_1,b_2,c,d_1,d_2$ such that the following are lines:
\[   \{\infty,a,b,c\}, \{\infty, a,a_1,a_2\}, \{\infty,b,b_1,b_2\}, \{ a,b,d_1,d_2\}. \]
Now observe that
\[
 g:=[\infty,a,b,\infty] = (a_1,a_2)(d_1,d_2)(b_1,b_2).
\]
Since $|\supp(g)|\leq6<2k$ (by Lemma \ref{l: support imprimitive}), we conclude that $g$ fixes all blocks of imprimitivity. Now consider the possibilities for $g$. By supersimplicity $d_1\neq\infty\neq d_2$, furthermore, since $\De$ contains no $\infty$-triangle, $a_i\neq b_j$ for all $1\leq i,j\leq 2$. Thus the possible coincidences between points, up to relabelling, are as follows:

\begin{itemize}
\item[(i)] if there are no coincidences, then $g= (a_1,a_2)(d_1,d_2)(b_1,b_2)$;
 \item[(ii)] if $a_1=d_1$, then $g=(a_1,a_2,d_2)(b_1,b_2)$;
 \item[(iii)] if $b_1=d_1$, then $g=(a_1,a_2)(b_1,d_2,b_2)$;
 \item[(iv)] if $a_1=d_1$ and $d_2=b_1$, then $g=(a_1,a_2,b_2,d_2)$.
\end{itemize}
Notice that in every case $a_1$ and $a_2$ are in the same block of imprimitivity. But, since $a$ was arbitrary, this implies that any pair of points 
$a_1$ and $a_2$ such that $\infty\in\overline{a_1, a_2}$ lie in the same block of imprimitivity. Now consider the line $\{\infty, c, c_1,c_2\}$.
As $\mathcal{D}$ contains no infinity triangle, $c_i\neq a_j,b_j$ for all $1\leq i,j,\leq 2$.  
We thus conclude that $a,b,c,a_1,a_2,b_1,b_2,c_1,c_2$ are nine distinct points that all lie in the same block
of imprimitivity, contradicting Lemma \ref{l: cool}.

\end{proof}

\begin{Lem}\label{l: infinity triangle}
If $n \geq 10$ then $k=2$ or $k=4$.
\end{Lem}

\begin{proof}
By Lemmas \ref{l: k not 3} and \ref{l: no infinity triangle} we may suppose that $k\geq 5$ and let $a,b,c,a_1,b_1,c_1$ be an $\infty$-triangle. Observe first that
\[
g:= [\infty,a,b,\infty]=(c,b_1)(e,f)(c,a_1)
\]
where $\{a,b,e,f\}$ is a line. Since $|\supp(g)|<6$ we observe that $g$ must fix all blocks and it is easy to calculate that $a_1,b_1$ lie in the same block $\Delta$, in fact the same is true of $c$ unless $\{a,b,a_1,b_1\}$ is a line. Since the set-up we have here is symmetrical we observe that
\begin{itemize}
 \item[(i)] $a_1,b_1,c_1$ all lie in the same block $\Delta$;
 \item[(ii)] if $\{a,b,a_1,b_1\}$ is not a line, then $c\in\Delta$;
 \item[(iii)] if $\{a,c,a_1,c_1\}$ is not a line, then $b\in\Delta$;
 \item[(iv)] if $\{c,b,c_1,b_1\}$ is not a line, then $a\in\Delta$.
 \end{itemize}

 Observe that if $\{a,b,a_1,b_1\}$ is a line, then
\[
h:= [\infty,a,a_1,\infty]=(a,a_1)\Big((b_1,c)(b,c_1)(b_1,b)(r,s)(b,c)(t,u)\Big)
\]
where $\{r,s,a,a_1\}$ and $\{t,u,a_1,\infty\}$ are lines.
If $h$ moves $\Delta$ then $b_1 \in \{t,u\}$ and $|\supp(h)|\leq 9 $, a contradiction, so $h$ fixes $\Delta$ and $a\in\Delta$. A symmetric argument also implies that $b\in\Delta$. Thus we deduce that 
\begin{itemize}
 \item[(i)] $a_1,b_1,c_1$ all lie in the same block $\Delta$;
 \item[(ii)] if $\{a,b,a_1,b_1\}$ is a line, then $a,b\in\Delta$;
 \item[(iii)] if $\{a,c,a_1,c_1\}$ is a line, then $a,c\in\Delta$;
 \item[(iv)] if $\{c,b,c_1,b_1\}$ is a line, then $b,c\in\Delta$.
 \end{itemize}
We conclude that, if either zero or two of the following sets -- 
\[\{a,b,a_1,b_1\}, \{a,c,a_1,c_1\}, \{c,b,c_1,b_1\}\]
-- are lines, then $\Delta \supseteq \{a,b,c,a_1,b_1,c_1\}$. Suppose then, that exactly one of these sets is a line, without loss of generality the first. In this case we know that $\Delta\supset\{a,b,a_1,b_1,c_1\}$. Then
\[
f:= [\infty,c,c_1,\infty]=(c,c_1)(a,b_1)(b,a_1)(r,s)(t,u)(b,a)(v,w),
\]
where
\[
 \{c,c_1,r,s\}, \{c,c_1,t,u\}, \{\infty,c_1,v,w\}
\]
are all lines. Observe that $c^f=c_1$ thus, if $f$ fixes $\Delta$, then $c\in\Delta$. Suppose that $f$ does not fix $\Delta$. Then we require that, for any $\delta\in\Delta$, $\delta^f\not\in\Delta$. But this is impossible since $a,b,a_1,b_1 \in \{r,s,t,u,v,w\}$, and $f$ cannot contain all elements in the image of $\Delta$ under $f$ in its support. We conclude that $\Delta \supseteq \{a,b,c,a_1,b_1,c_1\}$.

Now consider $e:=[\infty,b,d,\infty]$ where $d$ is any element that is not in $\Delta\cup\{\infty\}$. Note that, in particular, $d \notin \overline{\infty,b}$ (since these points lie in $\Delta$) and so $b^e=d$, $e$ does not fix $\Delta$ and $|\supp(e)| \geq 12$. Now observe that
\[
 e=(b,d)\Big((a,c_1)(c,a_1)(r,s)(t,u)(v,w)(x,y)\Big)
\]
where
\[
 \{b,d,r,s\}, \, \{b,d,t,u\}, \, \{\infty,d,v,w\}, \textrm{ and } \{\infty,d,x,y\}
\]
are all lines. Since $a,c,a_1$ and $c_1$ all lie in $\Delta$ it is clear that, to ensure $\Delta^e\cap\Delta=\emptyset$, these same four points must all lie in the set $\{r,s,t,u,v,w,x,y\}$. This implies, in particular, that $|\supp(e)|\leq 10$ which is a contradiction.

\end{proof}

\begin{Lem}\label{l: lambda 2 options}


$n \leq 13$. 

\end{Lem}
 \begin{proof}
Lemma~\ref{l: cool} implies that $n-1\leq 7\ell$. We apply Lemma~\ref{l: prim}.  Suppose that part (i) of Lemma~\ref{l: prim} holds, i.e.\ that if $a,c\in \Omega$ lie in the same block of imprimitivity, then $\infty\in\overline{a,c}$. Then Lemma~\ref{l: bound k above} implies that $k\leq 3$, Lemma~\ref{l: bound k below} implies $k\neq 2$ and Lemma \ref{l: k not 3} implies $k \neq 3$, a contradiction. 

Thus (ii) of Lemma~\ref{l: prim} holds and \begin{equation}\label{e: bounds} n\leq \min\{\frac{12\ell}{\ell-1}+1, 7\ell+1\}. \end{equation}

Suppose that $n>19$. If $\ell=2$, then by (\ref{e: bounds}) $n\leq 15$ which is a contradiction. Similarly if $\ell>2$, then 
\[n\leq \frac{18}{2}2+1=19\]
and, once again, we have a contradiction. Thus $n\leq 19$ and since any $2$-$(n,4,2)$ design satifies $n\equiv 1\pmod 3$,  it remains to consider separately the cases $n=16,19$.

Since $k \mid n-1$ it is an immediate consequence of Lemma \ref{l: infinity triangle} that $n \neq 16$. Thus $(n,\ell,k)=(19,9,2)$ so that $\ell>3$ and we have 
\[n\leq \frac{24}{3}2+1=17,\]
a contradiction.

 \end{proof}

\begin{proof}[Proof of Theorem C]
We apply Theorem~\ref{t: weak C} and the result holds, except if $\lambda=2$, ${13<n\leq 25}$ and $\pi_\infty(\De)$ is imprimitive. Now Lemma~\ref{l: lambda 2 options} implies that this exceptional situation is impossible, and we are done.
\end{proof}

\section{Open questions}\label{s: ext}

To conclude the paper we outline several avenues for further work.

\subsection{Possibilities for \texorpdfstring{$\pi_\infty(\De)$}{pi-infinity}}

We see no reason why Corollary \ref{c: prim} should not hold more generally. More precisely, we make the following conjecture:

\begin{Conj}\label{conjprim}
Let $\De=(\Omega,\B)$ be a supersimple $2$-$(n,4,\lambda)$ design. For all $\lambda > 0$ there exists $f(\lambda)$ such that either $\pi_\infty(\De)$ acts primitively on $\Omega \backslash \{\infty\}$ or $n < f(\lambda)$.
\end{Conj}

Observe that $\pi_\infty(\De)$ is generated by even permutations if $\lambda$ is odd and by odd permutations if $\lambda$ is even. By \cite[Corollary 3]{ls}, Conjecture \ref{conjprim} would imply the following:

\begin{Conj}\label{conjprim2}
Let $\De=(\Omega,\B)$ be a supersimple $2$-$(n,4,\lambda)$ design. For all $\lambda > 0$ there is some $g(\lambda)$ such that for all $n > g(\lambda)$,
 $$\pi_\infty(\De) \cong
\left\{
	\begin{array}{ll}
		S_{n-1}  & \mbox{if } \lambda \equiv 0 \pmod 2; \\
		A_{n-1}& \mbox{if } \lambda \equiv 1 \pmod 2.  
	\end{array}
\right.$$
\end{Conj}

Put another way, Conjecture \ref{conjprim2} would imply that for each $\lambda > 0$, there exist only finitely many hole stabilizers which are not alternating or symmetric groups. More ambitiously, we can ask:

\begin{Qu}\label{q: classify}
For what values of $n$ and $\lambda$ can one classify the groups $\pi_\infty(\De)$ (up to isomorphism) of all supersimple $2$-$(n,4,\lambda)$ designs $\De$?
\end{Qu}

Theorem C yields an affirmative answer to this question for all values of $n$ with $\lambda\leq 2$.

%
%
%

\subsection{Codes and Designs}
Recall that Conway et al.~ use the projective plane $\mathbb{P}_3$ to construct the perfect ternary Golay code.  They do
this by taking a subcode of the $\mathbb{F}_3$-rowspace of the incidence matrix of $\mathbb{P}_3$. Using the designs (other than $\mathbb{P}_3$)
described in this paper for which $\pi_{\infty}(\mathcal{D})$ is acting primitively on $\Omega\backslash\{\infty\}$ but does not contain $A_{n-1}$, we constructed the
following three codes in GAP:  
\begin{itemize}
\item[i)] $C$, the $\mathbb{F}_2$-rowspace of the incidence matrix of $\mathcal{D}$;
\item[ii)] $C^*$, the punctured code of $C$;
\item[iii)] $C_s$, the shortened code of $C$.
\end{itemize}
Because the lines of $\mathcal{D}$ consist of an even number of points, the codewords of $C$ have even weight.
In this case, the analogous code to the one constructed by Conway et al.~is the shortened code $C_s$.
Observe that, for the $2$-$(10,4,2)$ design,
the code $C_s$ is obtained by puncturing the code $C_p$ given in Example \ref{codeex}.
Certain parameters of these codes, for each design, are described in Table \ref{table2}. 
\begin{table}[ht]
\begin{center}
\begin{tabular}{ccc|c|c|c|c|c|c|}
\cline{4-9}
& & &\multicolumn{2}{ c| }{$C$}  & \multicolumn{2}{ c| }{$C^*$}&\multicolumn{2}{ c| }{$C_s$} \\ \cline{1-9}
\multicolumn{1}{ |c| }{$n$}&\multicolumn{1}{ c| }{$\lambda$}&\multicolumn{1}{ c| }{$k$}&$\rho$&$t$&$\rho^*$&$t^*$&$\rho_s$&$t_s$\\\hline
\multicolumn{1}{ |c| }{10}&\multicolumn{1}{ c| }{2}&5&3&3&2&2&3&5\\
\multicolumn{1}{ |c| }{16}&\multicolumn{1}{ c| }{3}&10&4&4&3&3&4&7\\
\multicolumn{1}{ |c| }{28}&\multicolumn{1}{ c| }{5}&21&3&3&2&2&3&5\\
\multicolumn{1}{ |c| }{36}&\multicolumn{1}{ c| }{9}&29&3&3&2&2&3&5\\\hline
\end{tabular}
\end{center}
\mbox{ } \newline
\begin{center}
\caption{Codes from $2$-$(n,4,\lambda)$ designs with primitive $\pi_{\infty}(\mathcal{D})$, not containing $A_{n-1}$.}\label{table2}
\end{center}
\end{table}

The parameter $k$ is the rank of the incidence matrix of $\mathcal{D}$ over $\mathbb{F}_2$. 
In each case $C$ is a $[n,k,4]$-code, $C^*$ is a $[n-1,k,3]$-code and $C_s$ is a $[n-1,k-1,4]$-code.
Moreover, $C$, $C^*$, $C_s$ has covering radius $\rho$, $\rho^*$, $\rho_s$, and external distance 
$t$, $t^*$, $t_s$, respectively. (Recall that the external distance of a linear code is the number of non-zero weights that appear
in the weight distribution of the dual code.) 

Completely regular codes have a high degree of combinatorial symmetry, and have been studied extensively (see \cite{delsarte, neum} and, more recently, 
\cite{nonantipodal, rho=2, binctrarb, kronprod,extendconst}).
Additionally, a certain family of distance regular graphs can be described as the coset graph of a completely regular code \cite[p.~353]{distreg}, 
so such codes are also of interest
to graph theorists. It is known that completely regular codes are necessarily uniformly packed (in the wide sense) \cite{distreg}.  
However, there are only a few examples of codes
known which are uniformly packed and not completely regular \cite{kronprod}. 

For each design, we see in Table \ref{table2} that $\rho=t$ and $\rho^*=t^*$, therefore $C$ and $C^*$ are uniformly packed (in the wide sense) \cite{remufp}.  Also, we 
observe that in lines $1$, $3$ and $4$ of Table \ref{table2}, the minimum distance of $C$ is equal to $2t-2$.  Therefore, as $C$
consists of codewords of even weight, for these lines in Table \ref{table2}, $C$ is completely regular \cite[p.~347]{distreg}. 
A result of Brouwer \cite{brou} implies that $C^*$ is completely regular in these cases also.  
We ask the following natural question.

\begin{Qu}
Let $\mathcal{D}$ be a supersimple $2$-$(n,4,\lambda)$ design such that $\pi_{\infty}(\mathcal{D})$
acts primitively but does not contain $A_{n-1}$.  Is the $\mathbb{F}_p$-rowspace (for some prime $p$) of the incidence matrix of $\mathcal{D}$
necessarily a completely regular and/or uniformly packed code in $\mathbb{F}_p^n$?
\end{Qu}

See \cite[Theorem C]{ggs} for some recent progress towards answering this question.




\subsection{The exceptional automorphism of \texorpdfstring{$S_6$}{S6}}

Recall that in the original paper \cite{Co} the authors `play the game' on the points of the projective plane $\mathbb{P}_3$ and show that the associated hole stabilizers are isomorphic to the Mathieu group $M_{12}$. 

By utilising the fact that $\mathbb{P}_3$ is \emph{self-dual} the authors are able to describe an alternative to the original game in which the roles of points and lines are reversed. By playing the two games simultaneously the authors are able to exhibit the outer automorphism of $M_{12}$ using the geometry of $\mathbb{P}_3$.

In the more general context of $2$-$(n,4,\lambda)$ designs, one cannot (obviously) pursue this idea, because the dual incidence system of a design is not necessarily a design, never mind a design isomorphic to the original. In particular, the only $2$-$(n,4,1)$ design that is self-dual is $\mathbb{P}_3$. 

Similarly, the only $2$-$(n,4,2)$ design that is self-dual is the unique $2$-$(7,4,2)$ design which, by Theorem C, has hole stabilizer isomorphic to $S_6$. Note that this exceptional property of the design (self-duality) is mirrored by an exceptional property of the associated hole stabilizer ($n=6$ is the only value for which $S_n$ has an outer automorphism). Thus it is natural to ask:

\begin{Qu}
 Can one exhibit an outer automorphism of $S_6$ via the geometry of the $2$-$(7,4,2)$ design?
\end{Qu}

\end{document}